\documentclass[letterpaper, 10pt, conference]{ieeeconf}

\IEEEoverridecommandlockouts 		

\renewcommand{\baselinestretch}{0.99}

\newtheorem{theorem}{Theorem}[section]
\newtheorem{lemma}[theorem]{Lemma}
\newtheorem{definition}{Definition}

\newtheorem{remark}{Remark}
\newtheorem{assumption}[theorem]{Assumption}

\newcommand{\mc}{\mathcal}

\newcommand{\vect}[1]{\operatorname{vec} \left( #1 \right) }

\newcommand{\real}{\mathbb{R}} 

\newcommand{\naturalset}{\mathbb{N}}
\newcommand{\realpos}{\mathbb{R}_{\geq 0}}

\newcommand{\tsp}{\mathsf{T}}

\newcommand{\Rc}{\supscr{R}{C}}
\newcommand{\Rs}{\supscr{R}{S}}
\newcommand{\onev} {\mathds{1}}

\newcommand{\map}[3]{#1: #2 \rightarrow #3}
\newcommand{\setdef}[2]{\{#1 \; : \; #2\}}

\newcommand{\supscr}[2]{{#1}^{\textup{#2}}}
\newcommand{\until}[1]{\{1,\dots,#1\}}

\usepackage{hyperref}
\usepackage{dsfont}
\usepackage{graphicx,color}
\graphicspath{{./IMG/}{images/}}
\usepackage{amsmath}
\usepackage{amssymb}
\usepackage{mathrsfs}
\usepackage[vlined,ruled]{algorithm2e}
\usepackage{subfigure}
\usepackage{url}
\usepackage{booktabs}
\usepackage{array}
\usepackage[table]{xcolor}
\usepackage{mathtools} 		
\usepackage{epstopdf}

\newcommand*{\QEDB}{\hfill\ensuremath{\square}}%

\title{\LARGE \bf 
Resilience of Traffic Networks with Partially Controlled  Routing}

\author{Gianluca Bianchin, Fabio Pasqualetti, and Soumya Kundu
  \thanks{This material is based upon work supported in part by  the United States 
  Department of Energy  contract DE-AC02-76RL01830 under the Control of 
  Complex Systems Initiative at Pacific Northwest National Laboratory, 
  and in part   by NSF award CNS-1646641.
	Gianluca Bianchin and Fabio Pasqualetti are with the Mechanical 
	Engineering Department, University of California at Riverside,
    \{\href{mailto:gianluca@engr.ucr.edu}{\texttt{gianluca}},
    \href{mailto:fabiopas@engr.ucr.edu}{\texttt{fabiopas\}@engr.ucr.edu}}.  
        Soumya Kundu is with the Pacific Northwest National Laboratory
        \href{mailto:soumya.kundu@pnnl.gov}{\texttt{soumya.kundu@pnnl.gov}}. }}

\begin{document}
\maketitle

\begin{abstract}
This paper investigates the use of Infrastructure-To-Vehicle (I2V) communication 
to 
generate routing suggestions for drivers in transportation systems, with 
the goal of optimizing a measure of overall network congestion.
We define link-wise levels of trust to tolerate the non-cooperative behavior of part 
of the driver population, and we propose a real-time optimization mechanism 
that adapts to the instantaneous network conditions and to sudden changes in the 
levels of trust.
Our framework allows us to quantify the improvement in travel time 
in relation to the degree at which drivers follow the routing suggestions.
We then study the resilience of the system, measured as the smallest change in 
routing choices that results in roads reaching their maximum capacity.
Interestingly, our findings suggest that fluctuations in the extent to which drivers 
follow the provided routing suggestions can cause failures of certain links. 
These results imply that the benefits of using Infrastructure-To-Vehicle 
communication come at the cost of new fragilities, that should be
appropriately addressed in order to guarantee the reliable operation of the 
infrastructure.
\end{abstract}

\section{Introduction}
Transportation systems are fundamental components of modern smart cities, and 
their effective and reliable operation are critical aspects to guarantee the 
development of quickly-growing metropolitan areas.
Recent advances in vehicle technologies, such as Infrastructure-To-Vehicle (I2V) 
communication and Vehicle-To-Vehicle (V2V) communication 
\cite{CW-AK-EV-AMB:17}, set out an enormous potential to overcome the 
inefficiencies of traditional  transportation systems.
Notwithstanding, the development efficient control algorithms capable 
of effectively engaging these capabilities is an extremely-challenging task
due to the tremendous complexity of the interconnections \cite{GB-FP:18},
that often results in suboptimal performance \cite{DAL-SC-RP:18}, and that can 
potentially generate novel fragilities \cite{VB-TH-GS:10,GB-FP:18a;G}.

In this paper, we discuss the use of I2V communication to partially influence the 
routing decisions of the drivers, with the goal of optimizing a measure of overall 
network congestion. 
We define link-wise levels of trust to tolerate the non-cooperative behavior 
of a certain ratio of the drivers, and we develop an optimization-based 
control mechanism to provide real-time routing suggestions based on the current 
state congestion levels.
Differently from traditional approaches for network routing design, our 
methods allow us to take into account quickly-varying traffic volumes, and do 
not require the knowledge of the traffic demands associated with every 
origin-destination pair.
%
%
Moreover, we study the impact of changes in routing that result in roads 
reaching  their maximum capacity, thus leading to traffic jams or cascading failure 
effects. 
We develop a technique to classify the links based on their resilience, and we study 
the fragility of the network against changes in routing.
Surprisingly, our findings demonstrate that networks where the routing is partially 
controlled by a system planner can be more fragile to traffic jam phenomena as 
compared to networks where drivers perform traditional selfish routing choices.

\noindent
\textbf{Related Work}
Routing decisions of traditional human drivers are \textit{non-cooperative}, namely, 
drivers act as a group of distinct agents that make selfish routing decisions with 
the goal of minimizing their individual delay \cite{JGW:52}.
The inefficiencies of such noncooperative behavior are often quantified 
through the \textit{price of anarchy} \cite{TR-ET:02,JRC-ASS-NES:04,GP:07}, a 
measure that captures the cost of suboptimality with respect to the societal 
optimal efficiency.
The availability of V2V and I2V has recently demonstrated the potential to influence 
the traditional behavior of drivers in a transportation system 
\cite{DAL-SC-RP:18,DAL-SC-RP:17}.
In particular, the control of the routing choices was proposed as a 
promising solution to improve the efficiency of the network \cite{DAL-SC-RP:17}
and to enhance its resilience \cite{GC-KS-DA-MAD-EF:13a}.
Differently from this line of previous work, this paper focuses on 
systems operating at non-equilibrium points, on tolerating the presence of 
non-cooperative driver behaviors, and on characterizing~the impact of controlled 
routing on the resilience of the system.

\noindent
\textbf{Contribution}
The contribution of this paper is fourfold. 
First, we formulate and solve an optimization problem to design optimal routing 
suggestions with the goal of minimizing the travel time experienced by all network 
users. 
The optimization problem incorporates link-wise trust parameters that describe the 
extent to which drivers on that link are willing to follow the suggested routing 
policy. 
Second, we develop an online update scheme that takes into account 
instantaneous changes in the levels of trust on the provided routing suggestions. 
Discrepancies between the modeled and actual trust parameters can be the result 
of quickly varying traffic demands, or can be the effect of selfish routing decisions.
Third, we study the resilience of the network, measured as the smallest change in 
the trust parameters that results in roads reaching their maximum capacity. 
We present an efficient technique to approximate the resilience of the network 
links, and we discuss how these quantities can be computed from the output of the 
optimization problem.
Fourth, we demonstrate through simulations that, although partially controlling the 
routing may improve the travel time for all network users, it also results in 
increased network fragility due to possible fluctuations in the trust parameters.

\noindent
\textbf{Organization}
The rest of this paper is organized as follows. 
Section~\ref{sec:problemFormulation} describes the dynamical network 
framework, and formulates the problem of optimal network routing with varying 
levels of trust.
Section~\ref{sec:routingControl} presents a method to numerically solve the 
optimization, and illustrates our real-time update mechanism.
Section~\ref{sec:resilience} is devoted to the study and characterization of the 
network resilience, while Section~\ref{sec:simulations} presents simulations 
results to validate our methods.
Finally, Section~\ref{sec:conclusions} concludes the paper.

\section{Problem Formulation}
\label{sec:problemFormulation}
We model a traffic network with a directed graph 
$\mc G =(\mc V, \mc E)$, where $\mc V = \until m$
denotes the set of nodes, and $\mc E = \until n \subseteq \mc V \times \mc V$ 
denotes the set of edges. 
Nodes of the graph identify traffic junctions, while edges identify sections of roads 
(links) that interconnect two junctions.
An element $(i,j) \in \mc E$ denotes a directed link from node $j$ to node $i$.
We associate to every link $i \in \mc E$ a dynamical equation of the form
\begin{align*}
\dot x_i = \supscr{f}{in}_i(x,t) - \supscr{f}{out}_i(x,t),
\end{align*}
where $t \in \realpos$, $\map{x_i}{\realpos}{\realpos}$ denotes the traffic density of link $i$, 
and $\supscr{f}{in}_i(x,t)$  and $\supscr{f}{out}_i(x,t)$ denote the inflow and 
outflow of the link, respectively. 
%
%
We assume that vehicle inflows enter the network at on-ramp links 
$\supscr{\mc E}{on}$, while vehicle outflows exit the network at off-ramp links 
$\supscr{\mc E}{off}$.
We denote by $\supscr{\mc E}{in}$ the set of internal links that are connected 
through junctions, and assume that $\supscr{\mc E}{on}$, $\supscr{\mc E}{off}$, 
and $\supscr{\mc E}{in}$ are  disjoint sets, with 
$\mc E = \supscr{\mc E}{on} \cup 
				\supscr{\mc E}{off} \cup
				\supscr{\mc E}{in}$
(see Fig.~\ref{fig:networkTopology} for an illustration).
%
The network topology described by $\mc G$ imposes natural constraints on 
the dynamics of the links, where flow is possible only between 
links that are interconnected by a node.
We associate to every pair of links a turning ratio $r_{ij} \in [0,1]$, 
describing the fraction of vehicles entering link $i \in \mc E$ after exiting 
$j \in \mc E$.
We combine the drivers turning preferences into a matrix 
$R=[r_{ij}] \in \real^{n \times n}$, where
\begin{align}
\label{eq:Rentries}
r_{ij} \in [0,1], && r_{ij} \neq 0 \text{ only if } (i,j) \in \mc E.
\end{align}
The conservation of flows at the junctions imposes the following constraints on 
the entries of $R$:
\begin{align}
\label{eq:RrowSum}
& \sum_{i} r_{ij} = 1, \text{ for all } j \in \mc E \setminus \supscr{\mc E}{off},
	\nonumber\\
& \sum_{i} r_{ij} = 0, \text{ for all } j \in \supscr{\mc E}{off}.
\end{align}

We let $\mc R_{\mc G}$ be the set of matrices
\begin{align*}
\mc R_{\mc G} = \setdef{ R=[r_{ij}]  \in \real^{n \times n} }{ r_{ij}
\text{ satisfy  \eqref{eq:Rentries}  and } \eqref{eq:RrowSum}},
\end{align*}
and let $n_r = \Vert R \Vert_0$ denote the number of nonzero entries in matrix $R$.
We assume that the vehicles routing is partially controllable, and denote by 
$\map{\sigma_i}{\realpos}{ [0,1]}$ the ratio of controllable vehicles that instantaneously occupy 
link $i$.
For every $i \in \mc E$, we assume that a fraction $(1-\sigma_i)$ of vehicles 
leaving $i$ will follow a selfish route choice $\supscr{r}{s}_{ij}$, for all $j \in \mc E$, 
while the remaining vehicles can be routed according to the routing 
decisions made by a system planner, namely $\supscr{r}{c}_{ij}$.
The parameter $\sigma_i$ can be interpreted as the (average) extent to which 
drivers follow the routing suggestion $\supscr{r}{c}_{ij}$.
We combine the selfish and controllable routing parameters into matrices
$\Rs \in \real^{n \times n}$ and $\Rc \in \real^{n \times n}$, respectively,
and decompose the matrix of turning preferences as
\begin{align*}
R = \Sigma \Rc + (I-\Sigma)  \Rs,
\end{align*}
where $\Sigma$ is a diagonal matrix 
$\Sigma = \operatorname{Diag} (\sigma_1, \dots, \sigma_n)$. 
Note that the graph topology and sparsity pattern of $R$ impose the following 
constraints on $\Rs$ and $\Rc$:
\begin{align*}
\Rs \in \mc R_{\mc G}, && \Rc \in \mc R_{\mc G}.
\end{align*}
We stress that in this work $\Rc$ is a design parameter containing the set of 
routing suggestions provided by the system planner to influence the drivers routing 
choices.

\begin{remark}{\bf \textit{(Selfish Route Choices)}}
Typically, the selfish behavior of drivers is captured by a 
\textit{Wardrop equilibrium} \cite{JGW:52}, that is, a configuration in which the 
travel time associated to any source-destination path chosen by a nonzero fraction 
of the drivers does not exceed the travel time associated to any other path. 
We remark that, in our settings, such equilibrium configuration is captured by the 
selfish routing matrix $\Rs$.
%
\QEDB
\end{remark}
\smallskip

\begin{figure}[t]
  \centering
    \includegraphics[width=.7\columnwidth]{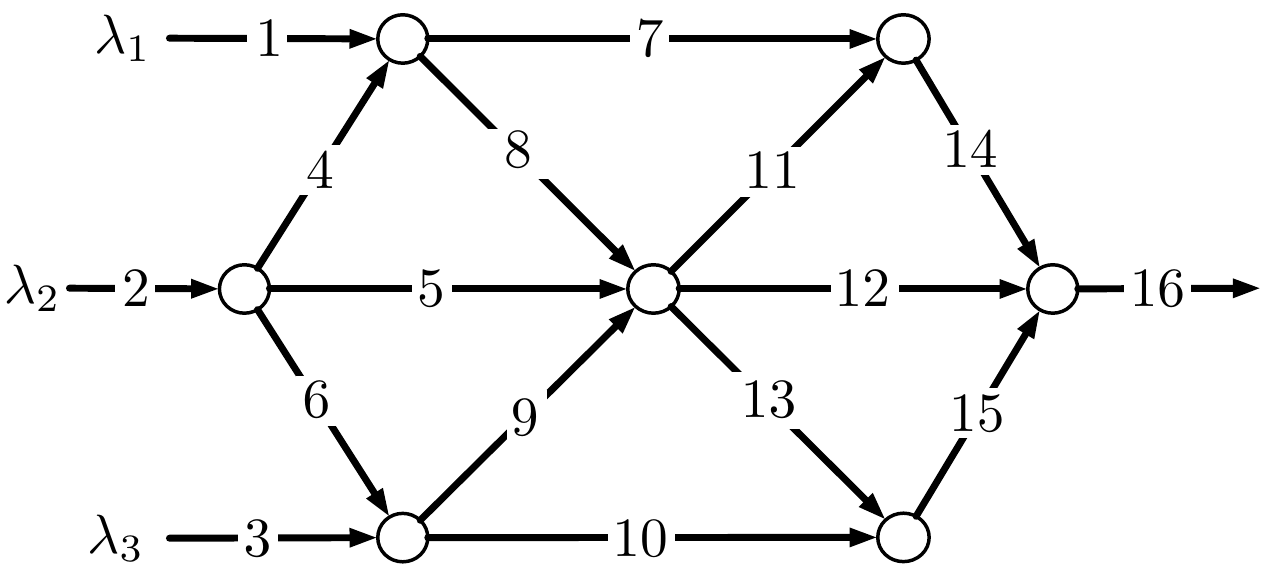}
\caption[]{Example of traffic network interconnection. 
For this network,
$\supscr{\mc E}{on} = \{1,2,3\}$, 
$\supscr{\mc E}{in} = \{4, \dots, 15\}$, and
$\supscr{\mc E}{off} = \{16\}$.}
  \label{fig:networkTopology}
\end{figure}


We adopt Daganzo's Cell Transmission Model 
\cite{CFD:95}, and model the physical characteristic of each link by a 
demand function $d_i(x_i)$ and a supply function $s_i(x_i)$, that represent upper 
bounds on the outflow and inflow of each link, respectively:
\begin{align}
\label{eq:boundsDemandSupply}
\supscr{f}{in}_i(x,t) \leq s_i(x_i), & & \supscr{f}{out}_i(x,t) \leq d_i(x_i).
\end{align}
For every link $i \in \mc E$ we let $B_i := \sup \{ x : s_i(x) > 0 \}$ denote its
saturation density, which corresponds to the jam density of the road.
We model on-ramps $i \in \supscr{\mc E}{on}$ as links 
with infinite supply functions $s_i(x_i) = + \infty$, and denote by 
$\lambda_i(t)$ the corresponding inflow rate.
Then, road inflows and outflows are related by means of the following 
equations 
\begin{align}\label{eq:inflowsOutflows}
\supscr{f}{in}_i (x,t) = \begin{cases}
\lambda_i(t), & i \in \supscr{\mc E}{on},\\
\sum_{j} r_{ij} \supscr{f}{out}_{j}(x,t), & 
		i \in \mc E \setminus \supscr{\mc E}{on},\\
\end{cases}
\end{align}
which capture the conservation of flows at the junctions.
We model the outflows  from the links as 
\begin{align}
\label{eq:f_out}
\supscr{f}{out}_{i}(x,t) = \kappa_i(x) d_i(x_i),
\end{align}
where $\kappa_i(x) \in [0,1]$ is a parameter that enforces the bounds
\eqref{eq:boundsDemandSupply} or, in other words,
guarantees that every outgoing link has adequate supply to accommodate the
demand of its incoming links.
Different models for $\kappa_i(x)$ have been proposed in the literature, 
and prevalent roles have been played by FIFO policies 
\cite{SC-MA:15} and proportional allocation rules 
\cite{EL-GC-KS:14}.
We combine the link dynamical equations with 
\eqref{eq:inflowsOutflows} and \eqref{eq:f_out} to derive the overall network
dynamics
\begin{align}
\label{eq:networkDynamics}
\dot x = (R-I) \supscr{f}{}(x,t) + \lambda,
\end{align}
where $I \in \real^{n \times n}$ is the identity matrix,
$x = [x_1~\dots~x_n]^\tsp$ is the vector of link densities, 
$\supscr{f}{} = [\supscr{f}{out}_1~\dots~\supscr{f}{out}_n]^\tsp$ is the 
vector of link outflows, and
$\lambda = [\lambda_1~\dots~\lambda_n]^\tsp$ denotes the vector of 
exogenous inflows, where we let $\lambda_i=0$ if $i \not \in \supscr{\mc E}{on}$.

We consider the network performance measured by the 
\emph{Total Travel Time} (TTT),
\begin{align*}
\text{TTT} &:= \int_0^{\mc H} x_1(t) + \dots + x_n(t) ~ dt,
\end{align*}
which is a measure of the delay experienced by all users \cite{GG-RH:06}, 
and we focus on the problem of designing the matrix of turning preferences in a 
way that
\begin{subequations}\label{opt:informal}
\begin{align}
 \min_{\Rc} \;\;\;\;\;\;\;\;\;
  		& \text{TTT}
  		\nonumber \\[.3em]
\text{subject to} \;\;\;\;
 	& \dot x = (R-I) \supscr{f}{}(x,t) + \lambda,
  			\label{opt:informal_constr_a}\\
	& R = \Sigma \Rc + (I-\Sigma)  \Rs ,
  			\label{opt:informal_constr_c}\\
  	& \Rc \in \mc R_{\mc G},   			
  			\label{opt:informal_constr_d}\\
  	& x \leq B,			
  			\label{opt:informal_constr_e}
\end{align}
\end{subequations}
where $\mc H$ is the control horizon,  $x(0) = x_0$ is the (given) network initial 
configuration, and  $B = [B_1 \dots B_n]^\tsp$ denotes the vector of jam densities.
From a real-time control and implementation perspective, solving 
\eqref{opt:informal} sets out a number of challenges.
First, the length of the optimization horizon $\mc H$ is a fundamental 
parameter that should be accurately chosen. One should chose $\mc H$ 
adequately large to include all relevant system dynamics, but unnecessarily large 
values of $\mc H$ can drastically increase the computational burden.
Second, rapid changes in traffic volumes and driver preferences require 
the development of control mechanisms that are capable to adapt in real-time 
to sudden variations of $\sigma$.
%
%
In fact, the performance of the optimization strongly depends on 
$\sigma$, and fluctuations in this parameter can lead to considerable variability in 
network performance and efficiency.

To study the effects of fluctuations in $\sigma$, in the second part of this paper we 
consider the problem of quantifying the fragility of the network against changes in 
the trust levels that result in links reaching their jam density.
We assume that a link irreversibly fails if it reaches its jam density, and argue 
that such failure may propagate in the network and potentially cause a cascading 
failure effect.
We measure the network resilience $\rho( \mc G, x_0)$ as the $L^1$-norm
of the smallest variation in $\sigma$ that results in such failure phenomena, that is,
\begin{align*}
\rho( \mc G, x_0) :=\;\;\;\;\; \min_{\tilde \sigma} \;\;\;\;\;\;\;\;\;
  		& \| \tilde \sigma - \sigma \|_1  \\[.3em]
\text{such that} \;\;\;\;  	& \dot x = (R-I) \supscr{f}{}(x,t) + \lambda,\\
	& R = \Sigma \Rc + (I-\Sigma)  \Rs,\\
 	& x_i \geq B_i, 
\end{align*}
where $x(0) = x_0$, $t \in [0, \mc H]$, and $i \in \until n$.

\section{Design of the Turning Preferences}
\label{sec:routingControl}
In this section we present a method to numerically solve the optimization problem 
\eqref{opt:informal}, and illustrate an online-update technique to address the 
control challenges outlined above.

\subsection{Computing Optimal Routing Suggestions}
We begin by recasting the optimization problem \eqref{opt:informal} in a way that 
allows us to numerically compute its solutions.
We perform three simplifying steps, described next.

First, in order to generate a tractable prediction of the time evolution of the 
network state, we discretize \eqref{eq:networkDynamics} by means of
the Euler discretization technique. 
We use a sampling time $T_s \in \real_{>0}$ that is chosen to guarantee the 
Courant-Friedrichs-Lewy assumption $\max_i \frac{v_i T_s}{L_i} \leq 1$ for all links
\cite{CFD:95}, where $v_i \in \realpos$ and $L_i \in \real_{>0}$ denote the 
maximum speed and the length of the section of road, respectively.
Let $\vect {R} = [r_{11} \dots r_{n1}~r_{12} \dots r_{nn}]$ denote the vectorization of
matrix $R=[r_{ij}]$, and let $t_k = k T_s $, $k \in \naturalset$.
Then, the time-evolution of \eqref{eq:networkDynamics} from 
$t_k$ to $t_{k+1} = t_k + T_s$ can be discretized as
\begin{equation}
\label{eq:discreteUpdate}
x_{k+1} = x_k + T_s ((R_k - I) f(x_k) + \lambda_k ) 
:= \mc F(x_k, r_k, \lambda_k),
\end{equation}
where $r_k = \vect{R_k}$.
We remark that the dependency on time of the routing matrix is the result of 
time-varying $\sigma$.

Second, we vectorize equation \eqref{opt:informal_constr_c} and let
\begin{align*}
r_k = (\Sigma_k^\tsp \otimes I) \supscr{r}{c} + 
		((I-\Sigma_k)^\tsp \otimes I) \supscr{r}{s} := 
		\Psi(\sigma_k, \supscr{r}{s}, \supscr{r}{c}),
\end{align*}
where $\supscr{r}{c} = \vect \Rc$, $\supscr{r}{s} = \vect \Rs$, the symbol 
$\otimes$ denotes the Kronecker product, and where we 
used the identity  $\vect {AXB} = (B^\tsp \otimes  A) \vect X$ for 
matrices $A$, $X$, and $B$ of appropriate dimensions.


Third, we observe that the Euler discretization technique employed in 
\eqref{eq:discreteUpdate} preserves the sparsity pattern of $\Rc$, and 
we rewrite the sparsity constraints \eqref{opt:informal_constr_d} as 
\begin{align*}
  	& \sum_{i} r_{ij}^\text{c} = b_j,   			
  	& 0 \leq \supscr{r_{ij}}{\!\!\!\!\!c} \leq 1,	& \;\;\; (i,j) \in \mc E,
\end{align*}
where 
$b_j = 1$ if $j \in \mc E \setminus \supscr{\mc E}{off}$, 
and $b_j=0$ if $j \in \supscr{\mc E}{off}$.

Finally, we recast the optimization problem \eqref{opt:informal}
by using the discretized dynamics as
\begin{subequations} \label{opt:discrete}
\begin{align}
\min_{\supscr{r}{c}} \;\;\;\;\;\;\;\;\;
  		& \sum_{k=1}^h  \onev^\tsp x_k			\nonumber\\[.3em]
\text{subject to} \;\;\;\;
 	& x_{k+1} = \mc F(x_k, r_k, \lambda_k), 							& k = 1, \dots , h,   	
 					\label{opt_discrete_a}\\
	& r_k = \Psi(\sigma_k, \supscr{r}{s}, \supscr{r}{c}),  		& k = 1, \dots , h,		
					\label{opt_discrete_b}\\
  	& \sum_{i} r_{ij} = b_j,   				& j = 1, \dots , n,					
  					\label{opt_discrete_c}\\
  	& 0 \leq \supscr{r_{ij}}{\!\!\!\!\!c} \leq 1,	& (i,j) \in \mc E,		
  					\label{opt_discrete_d}\\
  	& x_k \leq B, 																	& k = 1, \dots , h,
  					\label{opt_discrete_e}
\end{align}
\end{subequations}
where $\onev \in \real^n$ denotes the vector of all ones, and $h$ and $T_s$ are 
chosen so that $h T_s = \mc H$.
As discussed in e.g. \cite{AH-BDS-HH:05}, the constraints 
\eqref{opt_discrete_a} are often nonconvex in the decision variables.
Thus, the optimization problem \eqref{opt:discrete} is of the form of a 
nonconvex nonlinear programming optimization problem, over 
$n_r = \Vert R \Vert_0$ decision variables, and can be solved numerically through 
common nonlinear optimization solvers, such as interior-point methods 
\cite{AW-BA-LTB:06}.

\subsection{Online Update Mechanism}
In order to take into account for the quick variability of the parameter $\sigma$ 
and to deal with the considerable computational effort required to determine 
the solution to \eqref{opt:discrete}, we propose an adaptive control scheme that 
generates real-time updates based on the instantaneous changes in $\sigma$. 
The proposed adaptive mechanism is outlined in 
Fig.~\ref{fig:blockDiagramOnlineUpdate}, and is structured as follows.
We assume that a central processing unit is in charge of computing 
$R^{\text{C}*}(\sigma_0)$, that is, the solution to the optimization problem 
\eqref{opt:discrete} with a given (fixed) set of trust parameters $\sigma_0$.
The underlying choice for $\sigma_0$ can reflect 
the current network conditions, or can be dictated by the availability of historical 
data.
Moreover, we assume that the solution $R^{\text{C}*}(\sigma_0)$ is intermittently 
made available at time instants $t = k T_c$,  where $T_c \in \real_{>0}$ is the 
time required to solve the optimization.
We are interested in constructing an efficient mechanism to determine 
$R^{\text{C}*}(\sigma)$, the optimal solution to \eqref{opt:discrete}
with the instantaneous value of $\sigma$, by updating $R^{\text{C}*}(\sigma_0)$.
%
%
\begin{figure}[t]
  \centering
    \includegraphics[width=.9\columnwidth]{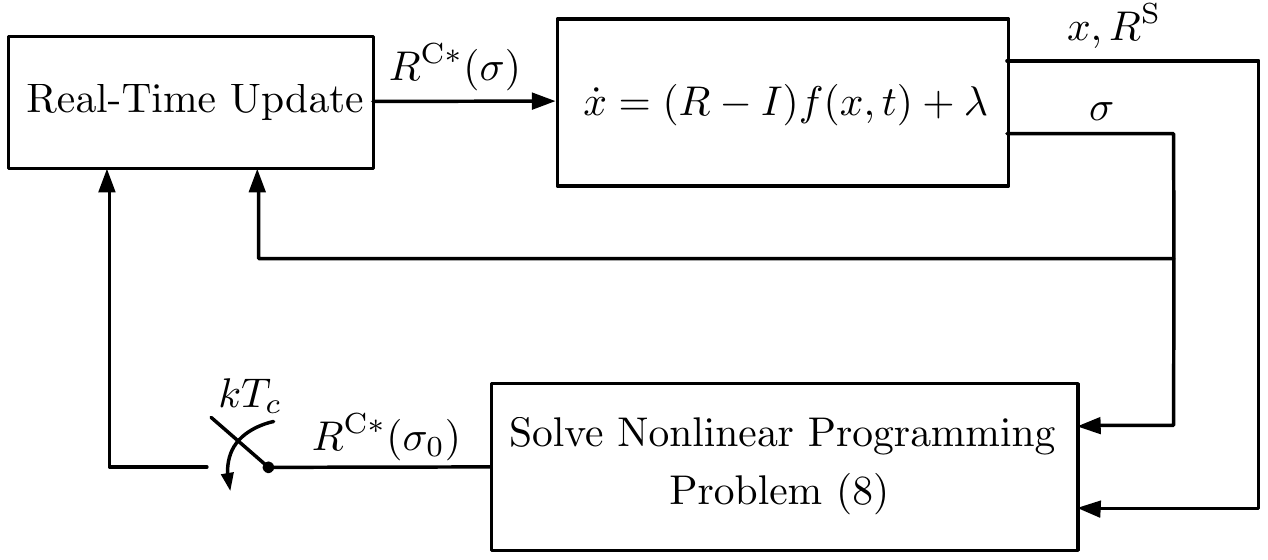}
\caption[]{Real-time update scheme.}
  \label{fig:blockDiagramOnlineUpdate}
\end{figure}
Our online update method is motivated by the fact that $\sigma$ is 
subject to small variations from the nominal value $\sigma_0$.
In fact, the following inequality follows from  \eqref{eq:Rentries}
\begin{align*}
\| \sigma - \sigma_0 \| \leq 
\Vert [1 \cdots 1]^\tsp\| = \sqrt{n_r}.
\end{align*}

Next, we derive our online update mechanism.
We denote in compact form by
\begin{align*}
f_0(\supscr{r}{c}, \hat x, \sigma)	= \sum_{k=1}^h \onev^\tsp x_k, \quad
g(\supscr{r}{c}, \hat x, \sigma) = 
\begin{bmatrix}
\supscr{r_{ij}}{\!\!\!\!\!c} - 1 \\ 
- \supscr{r_{ij}}{\!\!\!\!\!c}  \\
x_k - B
\end{bmatrix},\\
h(\supscr{r}{c}, \hat x, \sigma) = 
\begin{bmatrix}
x_{k+1} - \mc F(x_k, r_k, \lambda_k) \\ 
r_k - \Psi(\sigma, \supscr{r}{s}, \supscr{r}{c}) \\
\sum_{i} r_{ij} - b_j
\end{bmatrix},
\end{align*}
where $\hat x = [x_1^\tsp \dots x_h^\tsp]^\tsp \in \real^{n h}$ denotes the joint 
vector of model-prediction variables, and rewrite \eqref{opt:discrete} as
\begin{align}
 \min_{\supscr{r}{c}} \;\;\;\;\;\;\;\;\;
  		& f_0(\supscr{r}{c}, \hat x, \sigma)	\nonumber\\[.3em]
\text{subject to} \;\;\;\;
 	& g_i(\supscr{r}{c}, \hat x, \sigma) \leq 0, & i \in \until q, \nonumber\\
 	 	& h_j(\supscr{r}{c},\hat x,\sigma) = 0, & j \in \until p,
 	\label{opt:compact}
\end{align}
where we have made explicit the dependency of the optimization problem 
on the decision variables $\supscr{r}{c}$, on the prediction variables $\hat x$, and 
on the parameter $\sigma$.
To characterize the solutions to \eqref{opt:compact}, we compose the Lagrangian 
\begin{align*}
\mc L(\supscr{r}{c}, \hat x,\sigma, w, u) = 
f_0&(\supscr{r}{c}, \hat x, \sigma)	 +\\
&  u^\tsp g(\supscr{r}{c},\hat x,\sigma) 
+ w^\tsp h(\supscr{r}{c},\hat x,\sigma),
\end{align*}
where 
$u = [u_1 \dots u_q]^\tsp$  and $w = [w_1 \dots w_p]^\tsp$ are the vectors of 
Lagrange Multipliers, and we write the first order Karush-Kuhn-Tucker (KKT) 
conditions:
\begin{align*}
\nabla \mc L(r^{\text{c}*}, \hat x^*,\sigma_0, w^*, u^*) = 0, \nonumber \\
u_i g_i(r^{\text{c}*}, \hat x^*, \sigma_0) = 0, \nonumber\\
h_j(r^{\text{c}*}, \hat x^*, \sigma_0) = 0, 
\end{align*}
with the additional inequalities $u_i^* \geq 0$, and 
$g_i(r^{\text{c}*}, \hat x^*, \sigma_0) \leq 0$, where 
$\nabla = [\partial / \partial r^{\text{c}}_1 
		\dots 
		\partial / \partial r^{\text{c}}_{n_r}]^\tsp$ 
denotes the  gradient operator with respect to the decision variables 
$\supscr{r}{c}$. 
We denote the set of KKT equality conditions 
in compact form as 
\begin{align}
\label{eq:compactKKT}
F(r^{\text{c}*}, \hat x^*,\sigma_0, u^*, w^*) = 0,
\end{align}
and note that \eqref{eq:compactKKT} is an implicit equation that characterizes the 
optimal solutions to \eqref{opt:compact}. 
Finally, by letting $y = [\supscr{r}{c}(\sigma)~ u(\sigma)~w(\sigma)]$ and by 
assuming that \eqref{eq:compactKKT} holds for $\sigma$ near $\sigma_0$,
we compute the total derivative of the implicit function \eqref{eq:compactKKT} 
with respect to $\sigma$ to obtain the following relationship that holds at 
optimality:
\begin{align*}
M(\sigma) \frac{d y}{d \sigma} + N(\sigma) = 0,
\end{align*}
where the matrices 
$M (\sigma) = [\partial F_i / \partial y_j]$, $dy/d\sigma = [dy_i / d \sigma_j]$, and 
$N (\sigma) = [\partial F_i / \partial \sigma_j]$.
Finally, to formalize our online update rule we make the following classical 
assumption (see e.g. \cite{AVF-GPM:90}), which guarantees:
(i) that $r^{\text{c}*}$ is a local isolated minimizing point, 
(ii) the uniqueness of the Lagrange Multipliers, and 
(iii) the invertibility of matrix $M(\sigma_0)$.

\begin{assumption}{\bf \textit{(Second Order Minimizer Point)}}
\label{assumptions:KKT}

\noindent \textit{(Second-order KKT conditions)} 
The inequality 
$v^\tsp \nabla^2 \mc L(r^{\text{c}*}, \hat x^*,\sigma, w^*, u^*) v> 0$
holds for every vector $v \in \real^{n+m+p}$, $v \neq 0$, that satisfies
\begin{align*}
v^\tsp \nabla g_i(r^{\text{c}*}, \hat x^*, \sigma_0) & \leq 0, 
\text{ for all $i$ where } u_i^* = 0, \\
v^\tsp \nabla g_i(r^{\text{c}*}, \hat x^*, \sigma_0) &= 0, 
\text{ for all $i$ where } u_i^* > 0, \\
v^\tsp \nabla h(r^{\text{c}*}, \hat x^*, \sigma_0) &= 0.
\end{align*}
\noindent \textit{(Constraints independence)}
The vectors $\nabla g(r^{\text{c}*}, \hat x^*, \sigma_0)$ and 
$\nabla h(r^{\text{c}*}, \hat x^*, \sigma_0)$ are linearly independent.

\noindent \textit{(Strict complementary slackness)} If 
$g_i(r^{\text{c}*}, \hat x^*, \sigma_0) = 0$, then $u_i^* > 0$.
\QEDB
\end{assumption}
%
%
\begin{lemma}{\textbf{\textit{(Linear Update Rule)}}}
\label{thm:linearUpdate}
Let Assumption~\ref{assumptions:KKT} hold, 
let $r^{\text{c}*}(\sigma_0)$ denote a solution to \eqref{opt:discrete} with 
$\sigma = \sigma_0$, and let $ \eta := M^{-1}(\sigma_0) N(\sigma_0)$ 
be partitioned as
\begin{align*}
\eta = \begin{bmatrix}
\eta_{1} \\
\eta_{2} \\
\eta_{3}
\end{bmatrix},
\end{align*}
where $\eta_{1} \in \real^{n_r \times n} $,
$\eta_{2} \in \real^{q \times n} $, and
$\eta_{3} \in \real^{p \times n} $.
Then, 
\begin{align}
\label{eq:linearUpdate}
r^{\text{c}*}(\sigma)  =
r^{\text{c}*}(\sigma_0) + 
\eta_1 (\sigma - \sigma_0) + o(\| \sigma - \sigma_0\|^2 ).
\end{align}
\end{lemma}

\smallskip
We argue that the update \eqref{eq:linearUpdate} can be computed through 
simple vector multiplications, and thus is significantly more efficient than solving 
\eqref{opt:discrete}.
The accuracy of the linear approximation rule is numerically validated in 
Fig.~\ref{fig:linearApproximation}, which demonstrates the quadratic decay of  the 
approximation error as $\sigma$ approaches $\sigma_0$ (see 
Section~\ref{sec:simulations} for a thorough discussion).

\begin{figure}[bt]
  \centering
    \includegraphics[width=.85\columnwidth]{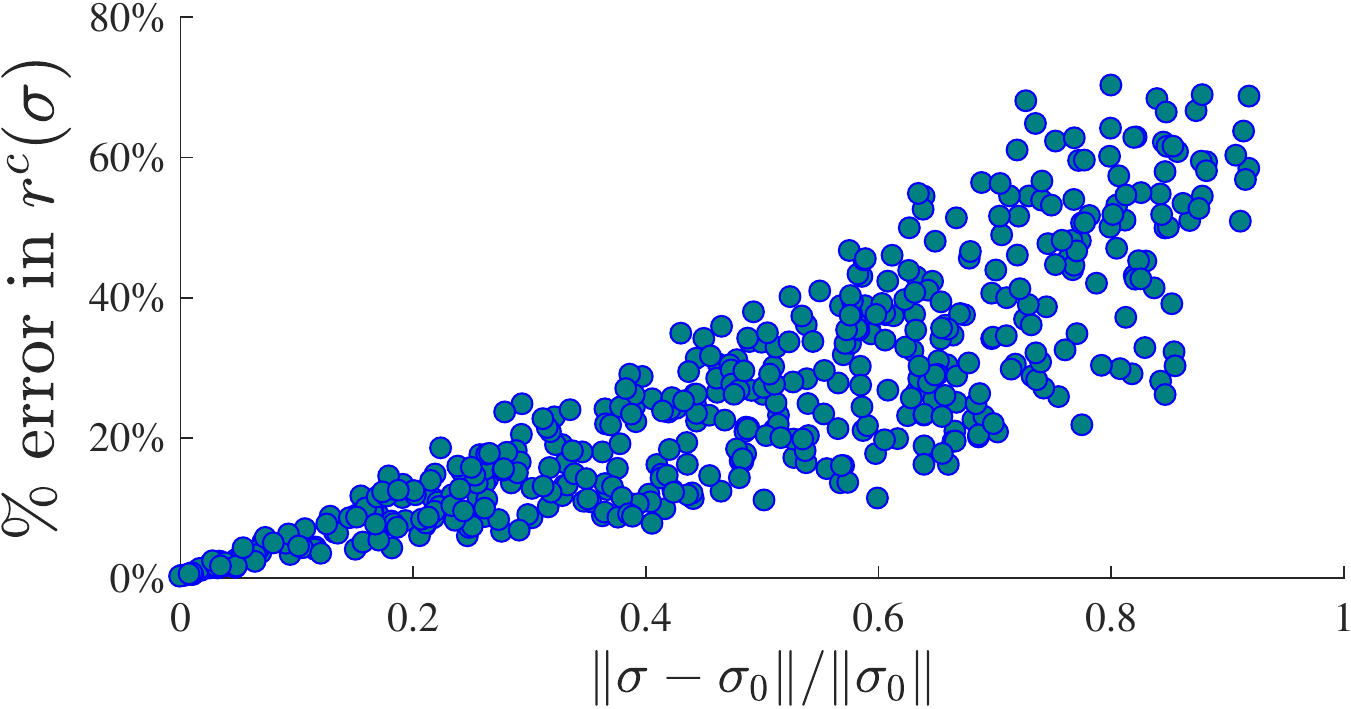}
\caption[]{Numerical validation of the update rule 
\eqref{eq:linearUpdate}. }
  \label{fig:linearApproximation}
\end{figure}

\section{Network Resilience}
\label{sec:resilience}
In this section, we study the resilience of the network against changes in the 
degrees of  trust of the drivers, 
and we illustrate a technique that allows us to classify the links 
in relation to their resilience properties.
%
%
We start with the following definition of margin of resilience of a network link.

\begin{definition}{\bf \textit{(Links Margin of Resilience)}}
\label{def:marginResilienceLinks}
Let $i \in \mc E$, and let $B_i$ be its jam density.
The margin of resilience of link $i$ is
\begin{align*}
\rho_i(x_0) :=\;\;\;\;\; \min_{\sigma} \;\;\;\;\;\;\;\;\;
  		& \| \sigma - \sigma_0 \|_1  \\[.3em]
  		\text{such that} \;\;\;\;  	& \dot x = (R-I) \supscr{f}{}(x,t) + \lambda,\\
	& R = \Sigma \Rc + (I-\Sigma)  \Rs,\\
 	& x_i \geq B_i, \text{ for some } t \in [0, \mc H].
\end{align*}
\QEDB
\end{definition}
In other words, the resilience of a certain link is defined as the smallest 
change in $\sigma$ that generates its jam failure.
Next, we present a lower bound on the margin of resilience of the links.
Our approach is based on the real-time control rule \eqref{eq:linearUpdate}, and on 
first-order approximations of the constraints.
\begin{theorem}{\bf \textit{(Lower Bound on Margin of Resilience)}}
\label{thm:lowerBoundResilience}
Let $i \in \mc E$, let 
$\mc F (x_k, r_k, \lambda_k) = 
[\mc F_1(x_k, r_k, \lambda_k) \dots \mc F_n(x_k, r_k, \lambda_k)]^\tsp$, and let
\begin{align*}
\Psi_i(r_k, x_k, \lambda_k, \sigma)  := 
\frac{\partial \mc F_i(x_k, r_k, \lambda_k)}{\partial \sigma} + 
\frac{\partial \mc F_i (x_k, r_k, \lambda_k)}{\partial \supscr{r}{c}} ~
\eta_1 ,
\end{align*}
where $\eta_1$ is defined in \eqref{eq:linearUpdate}.
Then,
\begin{align*}
\rho_i(x_0)  \geq  \min_k  \frac{B_i - \mc F_i (x_k, r_k, \lambda_k)}{
\left\Vert \Psi_i(k, \lambda, \sigma_0) \right\Vert_\infty} .
\end{align*}
\vspace{.1cm}
\end{theorem}

\begin{proof}
We first recast the notion of margin of resilience in terms of the discretized system 
\eqref{opt:discrete}. The margin of resilience of link $i$ is the smallest change 
$\| \sigma - \sigma_0 \|_1$ such that
\begin{align}
\label{eq:marginResDiscrete}
\mc F_i (x_k, r_k(\sigma), \lambda_k) \geq B_i,
\end{align}  
for some $k \in \until h$.
We then rewrite $\mc F_i (x_k, r_k(\sigma), \lambda_k)$ by taking its 
Taylor expansion for around $\sigma_0$
\begin{align*}
\mc F_i (x_k,  r_k(& \sigma) , \lambda_k) 
= \mc F_i (x_k, r_k(\sigma_0), \lambda_k) + \\
	&\underbrace{ \left. \frac{d \mc F_i }{d \sigma}(x_k, r_k(\sigma), \lambda_k) 
	\right \vert_{\sigma = \sigma_0} }_{\Psi_i(r_k, x_k, \lambda_k, \sigma)  }
 \delta_\sigma  	+  o(\| \delta_\sigma \|^2),
\end{align*}
where $\delta_\sigma = \sigma - \sigma_0$, and where we used the implicit 
differentiation rule to compute 
$ \Psi_i(r_k, x_k, \lambda_k, \sigma)  = 
	\frac{\partial \mc F_i }{\partial \sigma}  + 
	 \frac{d \mc F_i }{d \supscr{r}{c} }\; \frac{d \supscr{r}{c} }{d \sigma}$, 
with ${d \supscr{r}{c} }/{d \sigma} = \eta_1$.
By substituting into \eqref{eq:marginResDiscrete} and by rearranging the terms 
we obtain 
\begin{align*}
B_i - \mc F_i (x_k, r_k(\sigma_0), &\lambda_k)  + 
o(\| \delta_\sigma \|^2) \leq
\Psi_i(r_k, x_k, \lambda_k, \sigma) \delta_\sigma.
\end{align*}
Finally, we take the $L^1$-norm on both sides of the above inequality, 
which yields
\begin{align*}
\vert B_i &- \mc F_i (x_k, r_k(\sigma_0), \lambda_k)  +  o(\| \delta_\sigma \|^2) \vert \\
& \leq \left\vert \Psi_i(r^{\text{c}}, x_k, \lambda_k, \sigma) 
\delta_\sigma \right\vert \leq
\left\Vert \Psi_i(r^{\text{c}}, x_k, \lambda_k, \sigma)  \right\Vert_\infty
\left\Vert (\delta_\sigma) \right\Vert_1
\end{align*}
where we used Holder's inequality \cite{GHH-JEL-GP:88}.
To conclude, we iterate the above reasoning for all times $k \in \until h$, which 
yields the given bound for the margin of resilience and concludes the proof.
\end{proof}

We conclude this section by observing that the quantity 
$\Psi_i(r_k, x_k, \lambda_k, \sigma)$ is also a constraint of \eqref{opt:discrete}, and 
thus can be directly computed from the output of the optimization.
The tightness of the bound and the implications of the theorem are 
discussed in the next section.

\begin{figure}[t]
  \centering
    \includegraphics[width=.85\columnwidth]{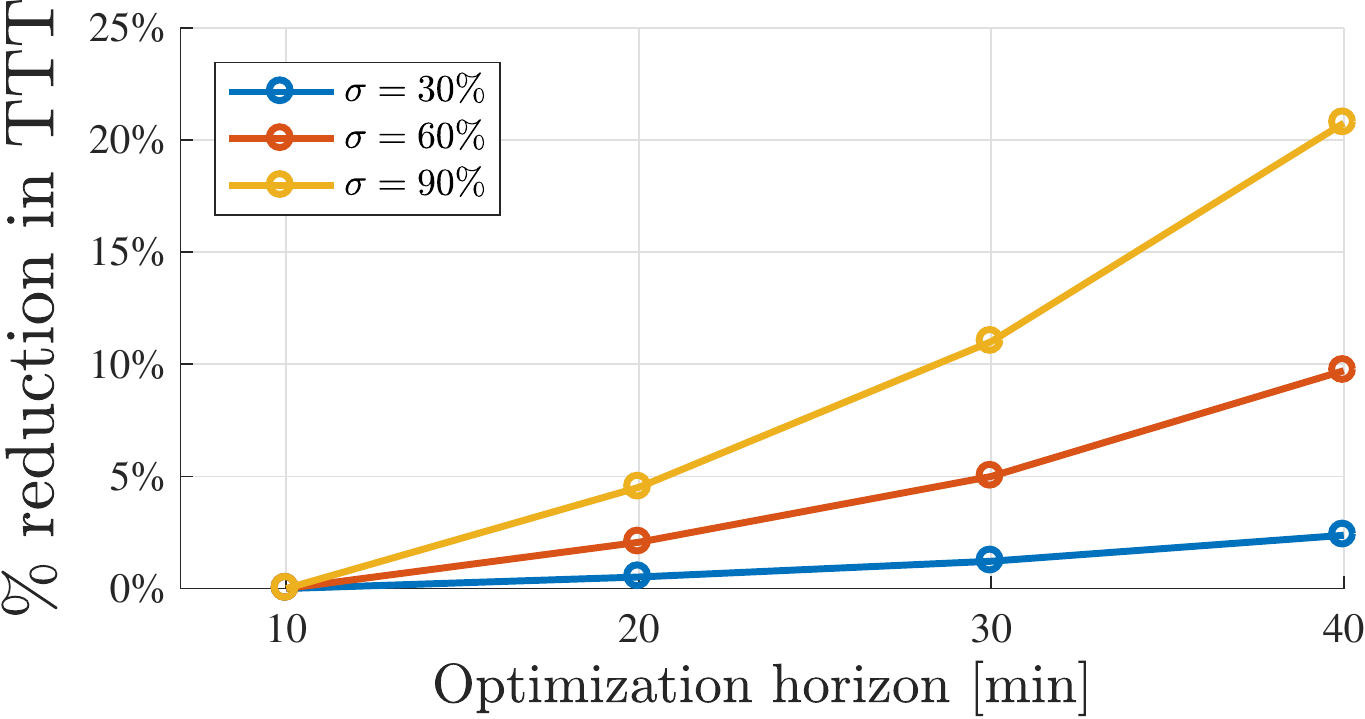}
\caption[]{Travel time reduction for degrees of trust and optimization 
horizons.}
  \label{fig:TTTimprovement_varySigma}
\end{figure}
\section{Simulation Results}
\label{sec:simulations}
This section provides numerical simulations in support to the 
assumptions  made in this paper, and includes discussions and 
demonstrations of  the benefits of the proposed methods.
We consider the network shown in 
Fig.~\ref{fig:networkTopology}, which comprises $ n =16$ links and $m = 7$ 
nodes.
Each link has capacity $B_i =B= 200 \text{veh}$, length 
$L_i =L= 5.25 \text{mi}$,  and velocity $v_i =v= 35 \text{mi/h}$.
For all $i$, we let $d_i(x_i) = v(1-\text{exp}(-a x_i))$, $a = 0.01$, and 
$s_i(x_i) = \frac{v}{L}(B-x_i)$ be the link demand and supply 
functions, respectively, and choose $\kappa_i(x)$ according to a
proportional allocation rule \cite{EL-GC-KS:14}.
We let $T_s = 0.15\text{h}$, and observe that $\max_i \frac{v T_s}{L} = 1$ satisfies 
the Courant-Friedrichs-Lewy assumption \cite{CFD:95}.
We let the network inflows be $\lambda_i = 10 \text{veh/min}$ for all 
$i \in \mc E^\textup{on}$, and assume that the density of the each link at time 
$t=0$ is $100\text{veh/mi}$, for all $i \in \mc E$.
The selfish turning preferences are chosen so that $r^\text{s}_{ij}$ is split 
uniformly between the outgoing links at every node.
Moreover, we assume $\sigma_i = \sigma$ for all $i$.

We begin by evaluating the benefits of partially controlling the network routing.
Fig.~\ref{fig:TTTimprovement_varySigma} illustrates the reduction in Total Travel 
Time in relation to different trust levels. 
The figure highlights that a consistent reduction in Total Travel Time is the 
combined result of significant levels of trust in the provided routing suggestions 
and of considerably-large control horizons.
Next, we investigate the network resilience in relation to changes in $\sigma$  
(Fig.~\ref{fig:linksResilience} and  \ref{fig:robustnessImprovement}).
To this aim, we show in Fig.~\ref{fig:linksResilience} the distance from jam density 
of every link in the network when drivers follow non-cooperative routing (i.e. 
$\sigma = 0$). 
Formally, this quantity is captured by the link residual capacity
\begin{align*}
\text{RC}_i :=  \min_k \;  \frac{B_i - \mc F_i(x_k, r_k, \lambda_k)}{B_i},
\end{align*}
which is a measure of the distance between the 
link density over time and its jam density $B_i$.
Note that, for the considered case study, all links operate with less than $30\%$ 
of their residual capacity.
The lower bound on the links margin of resilience 
(Theorem~\ref{thm:lowerBoundResilience}) is show in 
Fig.~\ref{fig:robustnessImprovement}.
Two important implications follow from the simulation results illustrated in 
Fig.~\ref{fig:robustnessImprovement}.
First, the trends observed in the figure support our 
observation that partially controlling the routing can result in increased fragility.
In fact, $\rho_i(x_0)$ for $\sigma_0=0$ is strictly larger that $\rho_i(x_0)$ for 
$\sigma_0=30\%$ for  all $i \in \{4, \dots 16 \}\setminus\{5,12\}$. 
Second, values of $\rho_i(x_0)$ greater than $100\%$ (observed, for instance, on
 link $i=16$) 
imply that no feasible change in $\sigma$ can lead to a jam failure of that link, 
while values of $\rho_i(x_0) < 100\%$ imply that there exists a feasible 
perturbation in $\sigma$ that results in jam-failures of that link. 
We note that the values reported in Fig.~\ref{fig:robustnessImprovement} are 
consistent with the considered network topology. 
In fact, the dynamics of link $i=16$ are independent of the routing choices 
performed by the drivers in the rest of the network.

\section{Conclusions}
\label{sec:conclusions}
This paper proposes a real-time optimization framework to design routing 
suggestions with the goal of optimizing the travel time experienced by all users in 
dynamical transportation systems that operate at non-equilibrium points. 
Our framework allows us to quantify the reduction in Total Travel Time in relation to 
different levels of trust on the provided routing suggestions, and to design a 
technique to study the resilience of the network links against fluctuations in 
the trust levels.
Our results reveal a tradeoff between efficiency and resilience in a transportation 
system, demonstrating that partially controlling the routing can reduce the margin 
of resilience of certain links.
Interesting aspects that require further investigation include extending the findings
to more general network topologies, and the design of incentive mechanisms 
to regulate and control the trust parameters.

\begin{figure}[t]
  \centering
    \includegraphics[width=.9\columnwidth]{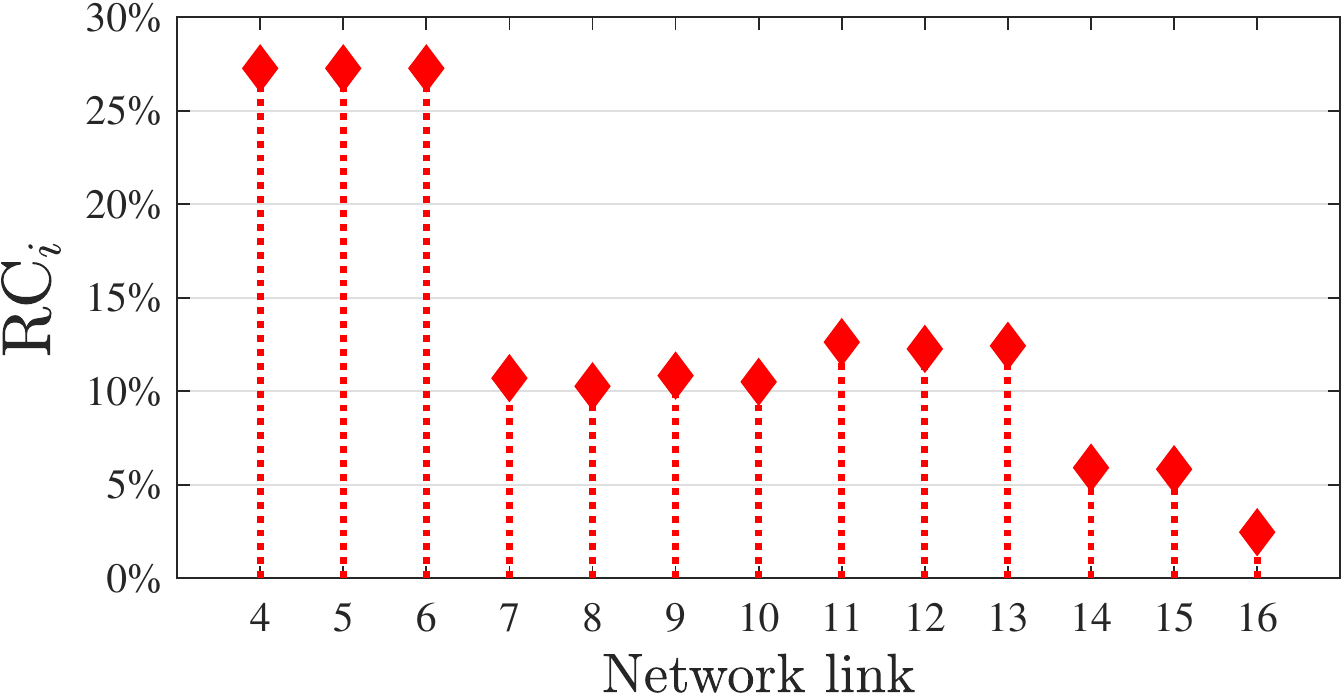}
\caption[]{Links distance from constraint violation (non-cooperative routing).}
  \label{fig:linksResilience}
\end{figure}
\begin{figure}[t]
  \centering
    \includegraphics[width=.9\columnwidth]{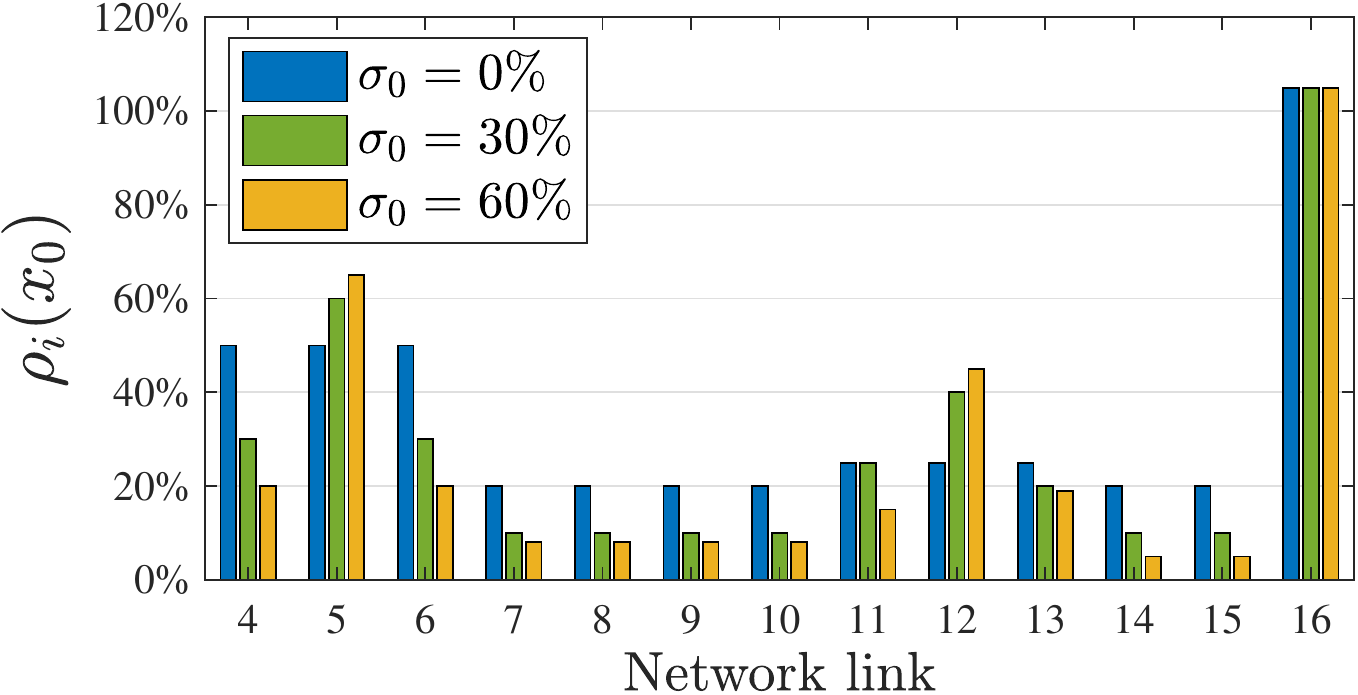}
\caption[]{Lower bound on links margin of resilience.}
  \label{fig:robustnessImprovement}
\end{figure}


\bibliographystyle{IEEEtran}
\bibliography{alias,Main,New,FP,BIB-GIANLUCA}
\end{document}